\documentclass[12pt]{amsart}
\usepackage{amssymb,amsmath}
\usepackage{geometry}    
\usepackage{mathrsfs}

\usepackage{graphicx}

\usepackage{vmargin}

\usepackage{yfonts}

\setmargrb{1in}{1in}{1in}{1in}

\newtheorem{theorem}{Theorem}[section]
\newtheorem{lemma}[theorem]{Lemma}
\newtheorem{proposition}[theorem]{Proposition}
\newtheorem{corollary}[theorem]{Corollary}
\theoremstyle{definition}
\newtheorem{definition}{Definition}
\newtheorem{remark}{Remark}

\newtheorem{conjecture}{Conjecture}
\newtheorem{problem}{Problem}

\newcommand{\ux}{\boldsymbol{\xi}}
\newcommand{\om}{\omega}
\newcommand{\wo}{\widehat{\omega}}

\newcommand{\Rm}{\mathbb{R}^m}

\newcommand{\R}{\mathbb{R}}

\newcommand{\w}{ \bold{w} }
\newcommand{\vv}{ \bold{v} }
\newcommand{\pp}{ \bold{p} }
\newcommand{\qq}{ \bold{q} }
\newcommand{\ttt}{ \bold{t} }

\newcommand{\x}{ \bold{x}  }

\begin{document}

\title[Dirichlet is not just bad and singular
in a class of fractals]{Dirichlet is not just bad and singular
	in many rational IFS fractals}

\author{Johannes Schleischitz}

\thanks{Middle East Technical University, Northern Cyprus Campus, Kalkanli, G\"uzelyurt \\
    johannes@metu.edu.tr ; jschleischitz@outlook.com}


\begin{abstract}
	For $m\ge 2$, consider $K$ the
   $m$-fold Cartesian product of the limit set of an IFS of two affine maps with rational coefficients. If the contraction rates of the IFS are reciprocals of integers, and $K$ does not degenerate to singleton,
   we construct vectors in $K$ that lie within the ``folklore set'' as defined by Beresnevich et al., meaning they are Dirichlet
   improvable but not singular or badly approximable (in fact 
   our examples are Liouville vectors). 
   We further address the topic of lower bounds for the 
   Hausdorff and packing dimension of these folklore sets within $K$, however we do not compute bounds explicitly. 
   Our class of fractals extends
   (Cartesian products of) classical missing digit fractals, for which analogous 
   results had recently been obtained. 
\end{abstract}

\maketitle

{\footnotesize{

Math Subject Classification 2020: 11J06, 11J13, 28A80}}


\section{Introduction and main results}  \label{s1.1}

\subsection{The ``folklore set'' within fractals}
We start by defining classes of real vectors in $\Rm$ 
according to their properties regarding rational approximation.
Denote by $\Vert \ux\Vert$ the distance of $\ux\in\R^m$ to the nearest integer vector with respect to the maximum norm. Following Davenport and Schmidt~\cite{ds},
we call $\ux\in\Rm$ Dirichlet improvable if 
for some $c\in(0,1)$ the system
\begin{equation}  \label{eq:eins}
1\leq q\le Q, \qquad \Vert q\ux\Vert < cQ^{-1/m}
\end{equation}
has an integer solution $q$ for all large $Q$. The terminology is rooted in Dirichlet's Theorem, which asserts 
that \eqref{eq:eins} is soluble for $c=1$ for any $\ux\in\Rm$.
We call $\ux$ singular if \eqref{eq:eins} has a solution
for arbitrarily small $c>0$ and $Q\ge Q_0(c)$.
We say $\ux$ is badly approximable if \eqref{eq:eins} has no solution
for all $Q$ and some $c>0$, or equivalently for some $c^{\ast}>0$
the estimate
\begin{equation}  \label{eq:ordappr}
 \Vert q\ux\Vert > c^{\ast} q^{-1/m}
\end{equation}
holds for any integer $q>0$. The setup \eqref{eq:eins} is usually called uniform approximation, whereas \eqref{eq:ordappr} is 
on ordinary approximation. We denote these classes of vectors, which are central objects of investigation in Diophantine approximation, by $Di_m, Sing_m$ and $Bad_m$ respectively. If $m=1$, then $Sing_1=\mathbb{Q}$ easily follows from an observation of Khintchine~\cite{khint},
and Davenport and Schmidt~\cite{sd} showed $Di_1=Bad_1$.
Let us now assume $m\geq 2$.
It is clear that $Sing_m\cap Bad_m=\emptyset$ and $Sing_m\subseteq Di_m$, on the other hand again Davenport and Schmidt~\cite{ds} showed that $Bad_m\subseteq Di_m$ as well.

These relations motivate to study the ``folklore set''
as defined in~\cite{beretc} (see also~\cite{marnat, js}) 
\[
\mathbf{FS}_m= Di_m\setminus (Sing_m\cup Bad_m).
\]
The main result of~\cite{beretc} constitutes that $\mathbf{FS}_m\ne \emptyset$
for $m\ge 2$.
The author provided a very different, constructive 
proof in~\cite{js}.
An advantage of the latter method
relevant to us in this paper is that the coordinates of $\ux$
may be chosen in classical missing digit
fractals, like the Cantor middle-third set, see the remarks below Theorem~\ref{1} below.
This paper aims to extend this result to a more general 
class of fractals.

We study classical fractals given as the attractor of an
iterated function system (IFS), see \S~\ref{IFS} below for details.
Fishman and Simmons~\cite{fs} (see also the subsequent
papers~\cite{erg, zh}) studied 
IFS of $J\ge 2$ affine
maps with rational coefficients, i.e.
\begin{equation} \label{eq:dDORN}
f_i(x)= c_i x+ d_i, \qquad 1\leq i\leq J, \; c_i\in(-1,1)\cap\mathbb{Q},\; d_i\in\mathbb{Q}.
\end{equation}
For our results below, we may
restrict to $J=2$ a fortiori, which we will henceforth assume
and write $f=f_1, g=f_2$. We can further assume 
that $f$ and $g$ have different fixed points, which translates into
\begin{equation}  \label{eq:4s} 
\frac{d_1}{1-c_1} \ne \frac{d_2}{1-c_2}.
\end{equation}
Indeed, otherwise the attractor degenerates to just this common fixed
point, a singleton.

\subsection{Our main result and a conjecture}

Given the results and proof strategy of~\cite{js}, it is plausible that the very mild,
necessary assumption \eqref{eq:4s}
suffices for 
the attractor of any IFS as in \eqref{eq:dDORN}
to contain vectors in $\mathbf{FS}_m$. Moreover, we believe 
that the property of not being 
badly approximable within the definition of $\mathbf{FS}_m$ can be considerably sharpened. Recall $\ux\in \Rm\setminus \mathbb{Q}^m$ is called Liouville vector if for arbitrarily large $t$ there is an integer solution to the estimate
\[
\Vert q\ux\Vert \leq q^{-t}.
\]
In the sequel, throughout we assume
\[
m\geq 2, \qquad K=C^m 
\]
where $C$ is the attractor of an IFS as above, 
and will identify the IFS with $C$, or $K$, when this is convenient. 
We would like to show the following idealistic claim.

\begin{conjecture}  \label{CC}
	For an IFS as in \eqref{eq:dDORN} satisfying \eqref{eq:4s}, there exist uncountably many Liouville vectors 
	$\ux\in K$ that are
	Dirichlet improvable but not singular. In particular
	\[
	\mathbf{FS}_m\cap K\neq \emptyset.
	\]
\end{conjecture}

Conjecture~\ref{CC} may hold
for irrational coefficients $c_i, d_i$ in \eqref{eq:dDORN} as well, however we are not able
to apply our method below in this case.
In the stated form with arbitrary rational coefficients it 
is still rather challenging, and we did
not succeed in proving it. 
Building up on the method from~\cite{js},
the problem turns out to become easier
when the contraction factors $c_i$ are reciprocals
of integers (remark: this restriction already occurred in~\cite[\S~4]{fs} and~\cite[Theorem~3.4]{erg} for different reasons), which just 
for simplicity of notation we assume to be positive.
Our main new result in this paper therefore specializes on 
the class of IFS
\begin{equation}  \label{eq:DORN}
f(x)= \frac{1}{b_1} \cdot x + \frac{r_1}{s_1}, \qquad
g(x)= \frac{1}{b_2} \cdot x + \frac{r_2}{s_2},
\end{equation}
where $b_i\ge 2, s_i>0, r_i$ are integers with $(r_i,s_i)=1$, and satisfying \eqref{eq:4s} which becomes
\begin{equation}  \label{eq:3s} \tag{i}
\frac{b_1r_1}{(b_1-1)s_1} \ne \frac{b_2r_2}{(b_2-1)s_2}.
\end{equation}
Indeed, we verify Conjecture~\ref{CC} for this smaller class of fractals.

\begin{theorem}  \label{1}
	For an IFS as in \eqref{eq:DORN} with restriction \eqref{eq:3s},
	the conclusion of Conjecture~\ref{CC} holds.
\end{theorem}

Theorem~\ref{1} extends 
the special case of missing digit fractals from~\cite[Theorem~3.1]{js},
that is for $C$ the set of real numbers whose
expansions to some given base $b\ge 3$ 
only
use digits within a two digit subset of $\{0,1,\ldots,b-1\}$.
This corresponds to the setting
	\begin{equation}  \label{eq:T}
b_1=b_2=s_1=s_2=b\ge 2,\qquad r_i\in\{ 0,1,\ldots,b-1\},\; r_1\ne r_2.
\end{equation}
An example of \eqref{eq:T} is the two-fold Cartesian product of the famous Cantor middle third set, obtained for the parameter choices
$m=2$, $b=3$ and $r_1=0, r_2=2$ in \eqref{eq:T}.
Note that the regularity condition \eqref{eq:4s}, or equivalently \eqref{eq:3s}, 
follows automatically from the setup \eqref{eq:T}. 
On the other hand,
presumably numbers in a general fractal $C$ derived from \eqref{eq:DORN}
have no pattern with respect to expansion in any base, however
results of this type may be hard to prove. We once again
stress that, maybe rather surprinsingly, \eqref{eq:3s} suffices for the claim of Theorem~\ref{1}, in particular usual regularity
assumptions like the open set condition need not be assumed.

Define the Dirichlet constant of $\ux\in\Rm$ as
\begin{equation}  \label{eq:DC}
\Theta(\ux):= \limsup_{Q\to\infty}\;\;\left( Q^{1/m}\cdot \min_{1\leq q\leq Q,\; q\in\mathbb{Z}}\Vert q\ux\Vert\right).
\end{equation} 
Then $\Theta(\ux)\in [0,1]$
and $\ux$ is Dirichlet improvable iff $\Theta(\ux)<1$ strictly,
and singular iff $\Theta(\ux)=0$. 
Our proof shows that, as in~\cite{js} for missing digit sets, 
it is possible to 
construct $\ux$ in a Cantor set as in Theorem~\ref{1} whose
Dirichlet constant differs from a given real number in $[0,1]$ 
at most by a fixed factor that depends on $m,b$ only.

We emphasize again that the method in~\cite{beretc} relying on Roy's Theorem~\cite{roy}
does not give any indication how to obtain the desired restriction
to fractals as in Conjecture~\ref{CC}.
Our method is based on the author's previous work in~\cite{js},
but the proofs presented below are self-contained.
If the rational contraction rates of $f,g$ are no longer reciprocals of integers, a supposedly severe problem 
in our approach occurs. In the ``metaresult'' Theorem~\ref{mull} below we establish a weaker conclusion if we permit one contraction rate to be arbitrary rational. Analyzing its proof (especially \eqref{eq:yi}),
the obstacle when trying to generalize to both rates being arbitrary rational becomes apparent. 
Finally, if some coefficients of the IFS are irrational, then the induced Cantor set may not contain any rational vector (see Boes, Darst and Erd\H{o}s~\cite{boes} for a rigorous proof of this fact for a similar type of Cantor sets), which completely disables our proof strategy. In particular the intrinsic folklore sets \eqref{eq:ito},
\eqref{eq:intr}, \eqref{eq:extr} in Theorem~\ref{2} below become empty.

\subsection{On intrinsic approximation}
Rewriting \eqref{eq:eins} by introducing the nearest
integers to $q\xi_j$, $1\le j\le m$, and dividing by $q$, gives 
rise to rational vectors $(p_1/q,\ldots,p_m/q)$ with common denominator
$q\le Q$ and distance $<cQ^{-1/m}q^{-1}$ to $\ux$, in maximum norm.
Let
us define $Di_m^{(K)}\subseteq Di_m\cap K$ the set of ``intrinsically Dirichlet improvable vectors'' 
with respect to $K$ for which for some $c\in(0,1)$ these rational vectors can be taken within the Cantor set $K$, for all large $Q$. The inclusion $Di_m^{(K)}\subseteq K$ hereby follows 
from the compactness of $K$. Define likewise the intrinsic sets $Sing_m^{(K)}\subseteq Sing_m\cap K$ and let $Bad_m^{(K)}$ be the 
set of vectors $\ux$ for which $\max_{i} \vert p_i/q-\ux\vert \gg_{\ux} q^{-1-1/m}$ with an implied constant depending on $\ux$ only, whenever
$(p_1/q,\ldots,p_m/q)\in K$. Then a fortiori
$Bad_m^{(K)}\supseteq Bad_m$. 
In fact $Bad_m^{(K)}\supseteq Bad_m\cup K^c$ since
again the compactness of $K$
implies $K^{c}\subseteq Bad_m^{(K)}$ (any chosen point of $K^{c}$ has positive distance from $K$, in particular this holds any 
rational vector of $K$ and thus $\max_{i} \vert p_i/q-\ux\vert \gg_{\ux} 1 \gg q^{-1-1/m}$ for $(p_1/q,\ldots,p_m/q)\in K$)
but this will not be of much relevance as we will restrict
to subsets of $Di_m^{(K)}\subseteq K$ below. Define likewise intrinsic Liouville vectors where the approximating rational vectors lie in $K$, and observe this set lies in the complement of $Bad_m^{(K)}$. We may now consider variants of ``intrinsic 
folklore sets'' where certain sets are replaced by their intrinsic
versions defined above. By the aforementioned inclusions, the smallest of these sets is
\begin{equation} \label{eq:ito}
\mathbf{FS}_m^{(K)}:= Di_m^{(K)}\setminus (Sing_m\cup Bad_m^{(K)})\subseteq \mathbf{FS}_m\cap K.
\end{equation}
The real vectors that we construct in
the proof of Theorem~\ref{1} indeed have this property of 
good rational approximants within $K$.
Consequently, Theorem~\ref{1} can be refined 
by means of the derived smaller intrinsic folklore sets as follows.

\begin{theorem}  \label{2}
	With notation and assumptions as in Theorem~\ref{1}, the set 
	\begin{equation} \label{eq:intr}
	Di_m^{(K)}\setminus Sing_m
	\end{equation}
	still contains uncountably many intrinsic Liouville vectors of $K$. 
	In particular $\mathbf{FS}_m^{(K)}\ne \emptyset$ and
 consequently also
	\begin{equation} \label{eq:extr}
	\mathbf{FS}_m^{(K)\ast}:= Di_m^{(K)}\setminus (Sing_m^{(K)}\cup Bad_m^{(K)})\ne \emptyset.
	\end{equation}
\end{theorem}

However, all proposed intrinsic sets 
are not very natural objects in the sense that the natural intrinsic Dirichlet function is no longer $Q^{-1/m}$. Hereby we mean the minimal
function  $\Phi: \mathbb{N}\to (0,\infty)$ so that \eqref{eq:eins} with right hand side replaced by $\Phi(Q)$ admits an intrinsic
solution for all $Q\ge 1$
and any $\ux\in\Rm$. This inspires the following problem.

\begin{problem}
	Does a variant of \eqref{eq:extr} for altered sets
	with respect to the natural intrinsic Dirichlet function hold?
\end{problem}

Let $m=1$, and consider $K=C$ derived from an IFS as in \eqref{eq:dDORN}, and assume the open set condition holds (see~\cite{fs} for instance). Then
the natural Dirichlet function
is of the form $\Phi(Q)=c(\log Q)^{-1/d}$
for $d$ the Hausdorff dimension of $C$ and some $c>0$, see~\cite{fs}. However, the precise
value of $c$ seems to be unknown. See also~\cite[Theorem~2.1]{erg}
for partial results for a similar class of Cantor sets in arbitrary dimension $m\ge 1$.

\subsection{Extensions: on exact approximation and metric theory} 
Theorem~\ref{1} can be generalized in certain directions similar as~\cite{js}.
Here we just provide more information on exact approximation and
metrical theory.
Firstly, we can prescribe uniform approximation with respect to a wide class of functions $\Phi$ (see~\cite{js}) up to a fixed factor. We
want to explicitly state a final result that 
captures a relaxed claim regarding 
power functions $\Phi$, and
even extends to slightly more general settings. For given $\ux\in\Rm$,
define its exponent of uniform simultaneous approximation $\wo(\ux)$ 
as the supremum of $t$ such that
\[
1\le q\le Q, \qquad \Vert q\ux\Vert \le Q^{-t}
\]
has an integer solution $q$ for all large $Q$. Then $\wo(\ux)\ge 1/m$
for any $\ux\in\Rm$ by Dirichlet's Theorem and $\wo(\ux)\le 1$ for any $\ux\in \Rm\setminus \mathbb{Q}^m$ by the observation of Khintchine~\cite{khint} recalled in the introduction.
We call $\ux\in\Rm$ totally irrational if it does not lie in a rational
hyperplane of $\Rm$, a natural restriction, and refer to the spectrum
of $\wo$ for $K$ for the set of all values $\wo(\ux)$ that occur for totally irrational arguments $\ux\in K$. For $u_1\ge 1$ an integer, define an IFS
\begin{equation} \label{eq:u1}
f(x)= \frac{u_1}{b_1} + \frac{r_1}{s_1}, \qquad 
g(x)= \frac{1}{b_2} + \frac{r_2}{s_2},
\end{equation}
where clearly we can assume $u_1<b_1$ and $(u_1,b_1)=1$. 
We must assume the analogue (generalization) of \eqref{eq:3s}, which 
simply becomes 
\begin{equation} \label{eq:v} \tag{ii}
\frac{b_1r_1u_1}{(b_1-1)s_1} \ne \frac{b_2r_2}{(b_2-1)s_2}.
\end{equation}
The fractals $K$ from Theorem~\ref{1} just respresent the special case $u_1=1$ and are thus contained in the following result.

\begin{theorem}  \label{mull}
	Let $K$ be derived from an IFS as in \eqref{eq:u1} with property \eqref{eq:v}.
	Then given $\omega\in [\frac{1}{m}, \frac{1}{m-1}]\cup \{1\}$, there exist 
	totally irrational
	Liouville vectors in $K$ with $\wo(\ux)=\om$. In particular, the spectrum of $\wo$ for $K$ contains $[\frac{1}{m}, \frac{1}{m-1}]\cup\{1\}$.
\end{theorem}

\begin{remark}
	To obtain the claim for $\omega=1$ only, we can take any IFS as in \eqref{eq:dDORN}.
\end{remark}

The accordingly modified claims of
Theorem~\ref{2} on intrinsic approximation apply as well.
For $\om<\frac{1}{m-1}$ the condition of $\ux$ being totally irrational
is automatically satisfied, so in particular
in context of Theorem~\ref{1}.   
Theorem~\ref{mull} complements results by Kleinbock, Moshchevitin and Weiss~\cite[Theorem~1.6]{kmw}, who
provide the weaker conclusion that the spectrum of $\wo$ has non-empty intersection with $[\frac{1}{m-1},1]$,
but for a larger class of fractals (still their result seems
not to cover IFS as in \eqref{eq:dDORN} but with irrational coefficients $c_i, d_i$).
It is likely that small twists of the proofs 
enable one to extend the claim to the entire interval $[1/m,1]$, 
as for $m=2$.
Note that upon dropping the totally irrational condition, this an easy consequence of the original result. We may consider
$\tilde{m}:=\lceil \omega^{-1}\rceil$ and obtain totally irrational vectors $\tilde{\ux}\in \R^{\tilde{m}}$ as in Theorem~\ref{mull}.
Then it suffices to take the remaining $m-\tilde{m}\ge 0$ coordinates 
of $\ux$ arbitrarily within the $\mathbb{Q}$-span
of $\{ \tilde{\xi}_1,, \ldots, \tilde{\xi}_{\tilde{m}},1\}\in \R^{\tilde{m}+1}$. Note further that in the special case 
of missing digit Cantor sets,
the spectrum of $\wo$ is indeed $[1/m,1]$, as follows
easily from the more general results in~\cite[Theorem~2.5]{ich2013}.

Regarding metrical theory of the
sets $\mathbf{FS}_m \cap K$, our method alludes
that lower bounds on Hausdorff dimension and possibly also 
packing dimension can be 
derived, by a similar strategy as in~\cite{js}: If $b_1=b_2$,
we can arbitrarily redefine
the digits of the words $\w_j$ defined in~\S~\ref{s} below
at positions
in the same type of long intervals as in~\cite[\S~6]{js},
without the induced real vectors leaving the set. 
When $b_1\neq b_2$, one needs to be more careful in view of certain restrictions that enter in order to preserve an analogue of \eqref{eq:ufoA} below, but the basic idea remains the same.
What remains to be done is to estimate from below the dimensions of the subsets of $\mathbf{FS}_m \cap K\subseteq \Rm$ induced by these digital patterns. We believe
this should be feasible by a similar strategy as in~\cite{js}, 
at least for the Hausdorff dimension, possibly again with the aid of a principle from Falconer's book~\cite[Example~4.6]{falconer}. We 
conjecture that this bound is independent of the shifts,
and therefore if $b_1=b_2$ agrees with the bounds in the case of missing digit sets from~\cite{js}.

\subsection{A remark on more general fractals}

We finally notice that attractors of certain IFS consisting of rational, affine maps defined on $\mathbb{R}^m$ do not contain any element in $\bold{FS}_m$. Indeed, consider for example $K$ the attractor of
\[
f(\x)= \frac{1}{2} \x, \qquad g(\x)= \frac{1}{3} \x + \frac{1}{4}\cdot 
(1,1,\ldots,1)^{t}.
\]
Then $K$ lies in the one-dimensional rational subspace defined by $x_1=x_2=\cdots=x_m$, so any $\ux\in K\setminus \mathbb{Q}^m
$ has Dirichlet exponent $\wo(\ux)=1$, more precisely
\[
1\leq q\leq Q, \qquad \Vert q\ux\Vert \le Q^{-1}
\]
has an integer solution $q$ for all $Q>1$.
In particular, all elements of $K$ are (very) singular 
as soon as $m\ge 2$.
So it seems reasonable to consider Cartesian products of one-dimensional objects as in our results to avoid obstructions of this kind.

\subsection{Notation and some IFS theory}  \label{IFS}

We use $A\asymp B$ to denote $A\ll B\ll A$,
with Vinogradov's notation $A\ll B$ meaning that $A\leq cB$ for
some fixed $c>0$. We write $\mathbb{P}$ for the set of prime numbers.
We write $\lfloor x\rfloor$ for the largest integer less than or equal to $x\in\R$
and $\lceil x\rceil$ for the smallest integer larger or equal to $x$.
Denote as usual by $v_p(d)\in\mathbb{Z}$ the multiplicity of the prime $p$ in a rational number $d$, which is negative if $p$ divides the denominator (assuming $d$ is reduced), where we let $v_p(0)=\infty$.

Any IFS consisting of contracting maps $f, g$ on $\R$ induces an attractor $C\subseteq \mathbb{R}$ 
defined as the set of points
\[
\xi= \w(0)=\lim_{k\to\infty} w_1\circ w_2 \circ \cdots w_k(0),
\]
obtained from infinite words $\w=w_1w_2\cdots$ 
with $w_i\in \{f,g\}$. We call $w_i$ digits
and the digit string $\w=w_1w_2\cdots$ an address of $\xi\in C$, and
omit the symbol $\circ$ occasionally. Any $\xi\in C$ has at least one address, but possibly it is not unique. Further we recall that
a sequence of words $(\w_k)_{k\ge 1}$ converges to a word $\w$ 
by definition if any finite prefix of $\w$ coincides with the according string of $\w_k$
for all large enough $k$. It is easy to see that $\w_k\to \w$
implies $\w_k(0)\to \w(0)$ in $\R$.
We write $|\w|$ for the length of a finite word $\w$.
 If $\w_1$ is a finite word and
$\w_2$ any word, we just write $\w_1\w_2$ for the concetanation
of the words, i.e. reading first the digits of $\w_1$
followed on the right by the digit string of $\w_2$.

\section{Proof of Theorem~\ref{1} }

The proof generalizes the ideas from~\cite{js}, however
some technical hurdles have to be mastered.
Start with $c\in(0,1)$ small enough. The goal is to construct 
a Liouville vector $\ux\in\Rm$
with Dirichlet constant $\Theta(\ux)\asymp c$,
see definition \eqref{eq:DC},
where the implied
constants only depend on the IFS coefficients only. 
If we choose $c$ smaller
than the inverse of the implied constant for the upper estimate, then $\ux$
is Dirichlet improvable, and
the other claimed properties not singular and not badly approximable 
are obvious. 
We will first construct suitable $\ux\in K$ in~\S~\ref{s}. 
Then we proceed to present some
auxiliary results in~\S~\ref{2.1}-\S~\ref{2.4} whose proofs only
require elementary number theory and basic IFS theory.
The core of the proof  in~\S~\ref{SIA},~\ref{ASI} is to
finally verify $\Theta(\ux)\asymp c$, however given the auxiliary results
this works very similar as for the special cases of missing digit fractals in~\cite{js}.

\subsection{Construction of suitable $\ux$}  \label{s}

We essentially follow the construction for Cartesian products of missing digit fractals in~\cite{js}, where reading a base $b$ digit $1$ resp. $0$
here becomes reading contraction $f$ resp. $g$. 
In~\cite{js} the components of $\ux$ have base $b$ digit $1$ at isolated positions between long strings
of $0$ digits. In our more general setup, we take blocks of $N$ consecutive digits $f$ between long strings of digits $g$ instead, for large enough $N$ depending on the IFS.
This  will be reflected in the 
periodic suffix words $\pp, \qq$ defined below.

Let $N$ be a large integer, a lower bound to be defined in~\S~\ref{2.1} below.
Define the ultimately periodic words 
\[
\pp=f^{N}g^{\infty}, \qquad \qq=g^N \pp=g^N f^N g^{\infty}.
\]
Let $(M_i)_{i\ge 1}$ be any fast increasing (lacunary) 
sequence of positive integers.

We define $m+1$ increasing positive integer
sequences $(\textswab{f}_{k})_{k\ge 0}$ 
and $(\textswab{g}_{j,k})_{k\ge 0}$,
for $1\le j\le m$. The first
is simply given by
\begin{equation}  \label{eq:ID}
\textswab{f}_{k}= Nk, \qquad \; k\ge 0.
\end{equation}
For $\textswab{g}$ sequences, define the initial terms as
\[
\textswab{g}_{j,0}=0, \qquad 1\le j\le m
\]
and complete the sequence for $j=1$ via
\begin{equation} \label{eq:folgt}
\textswab{g}_{1,k}=M_k, \qquad  k\ge 1.
\end{equation}
For the remaining $j$ and $k\ge 1$, we choose
$\textswab{g}_{j,k}$ so that they satisfy
\begin{equation} \label{eq:ple}
b_1^{ \textswab{f}_{k} } b_2^{ \textswab{g}_{j,k} } \asymp (b_1^{ \textswab{f}_{k} } b_2^{ \textswab{g}_{1,k} } )^j, \qquad 2\leq j\le m-1, 
\end{equation}
and
\begin{equation}  \label{eq:simply}
b_1^{ \textswab{f}_{k} } b_2^{ \textswab{g}_{m,k} } \asymp c^m(b_1^{  \textswab{f}_{k} } b_2^{ \textswab{g}_{1,k} } )^m,
\end{equation}
with some absolute implied constant depending on the IFS only
but not on $k$. 
It is easy to see this can be done, we may take
\begin{equation}  \label{eq:simpli}
\textswab{g}_{j,k}= \left\lfloor \frac{\log (b_1^{(j-1)\textswab{f}_{k}} b_2^{j\textswab{g}_{1,k}}) }{ \log b_2}\right\rfloor, \quad (2\le j\le m-1), \quad
\textswab{g}_{m,k}= \left\lfloor \frac{\log (c^m b_1^{(m-1)\textswab{f}_{k}} b_2^{m\textswab{g}_{1,k}}) }{ \log b_2}\right\rfloor.
\end{equation}
Then it is easy to see that if $(M_i)_{i\ge 1}$ increase fast enough then
\[
\eta_{j,k}:=\; \textswab{g}_{j,k} - \textswab{g}_{j,k-1} > 0, \qquad 1\le j\le m, \; k\ge 1.
\]
Define $m$ sequences of finite words with initial terms $\vv_{j,0}=g^N$ for $1\leq j\le m$ and
\[
\vv_{j,k}= g^N f^N g^{\eta_{j,1}} f^N g^{\eta_{j,2}} \cdots g^{\eta_{j,k-1}}f^N g^{\eta_{j,k}}, \qquad 1\le j\le m, \; k\ge 1.
\]
Define also
their prefix sequences when omitting the last $N$ digits so that
$\ttt_{j,0}=\emptyset$ and
\[
\ttt_{j,k}= g^N f^N g^{\eta_{j,1}} f^N g^{\eta_{j,2}} \cdots g^{\eta_{j,k-1}}f^N g^{\eta_{j,k}-N}, \qquad 1\le j\le m, \; k\ge 1.
\]
By the fast increase of $M_k$ we
have 
\begin{equation} \label{eq:agha}
|\vv_{1,k}|<|\vv_{2,k}|<\cdots <  |\vv_{m-1,k}|< |\vv_{m,k}|< |\vv_{1,k+1}|, \qquad k\ge 1,
\end{equation}
and likewise for $\ttt_{.,.}$.
Derive sequences of ultimately periodic infinite words by
\[
\w_{j,k}= \ttt_{j,k}\qq= \vv_{j,k}\pp= g^N f^N g^{\eta_{j,1}} f^N g^{\eta_{j,2}} \cdots g^{\eta_{j,k}} f^N g^{\infty}, \qquad 1\le j\le m,\; k\ge 0.
\]
The initial terms are $\w_{j,0}=\qq$
and they converge to infinite words 
\begin{equation}  \label{eq:HH0}
\w_{j}= g^N f^N g^{\eta_{j,1}} f^N g^{\eta_{j,2}} \cdots , \qquad 1\le j\le m.
\end{equation}
Then finally we define the components of $\ux=(\xi_1,\ldots,\xi_m)$ to be
\[
\xi_j= \w_j(0)= \lim_{k\to\infty} \w_{j,k}(0), \qquad 1\le j\le m.
\]
$\bold{Observation}$:
For $k\ge 0$, the integers $\textswab{f}_{k}$ resp. $\textswab{g}_{j,k}$ count the occurrences of $f$ resp. $g$ in the words $\ttt_{j,k}$. In particular
\begin{equation}  \label{eq:AB1}
|\vv_{j,k}|-N= |\ttt_{j,k}|= \textswab{f}_{k}+ \textswab{g}_{j,k}, \qquad 1\le j\le m, \; k\ge 0.
\end{equation}



\subsection{A rational in $C$ with denominator divisible by large $b_1b_2$ powers }  \label{2.1}

For the proof of the crucial Lemma~\ref{wicht} in~\S~\ref{2.4} below,
we first require the existence of rational numbers in $C$ whose denominator is divisible by an arbitrarily large given power of $b_1b_2$.
We establish this auxiliary result in this section.

\begin{lemma}  \label{lemur}
	Let $C$ be induced by an IFS as in \eqref{eq:DORN} with restriction
	\eqref{eq:3s}.  
	Let $\ell\ge 0$ be an integer. Then for any 
	large enough positive integer $N$,
	then writing $\qq=g^{N}f^{N}g^{\infty}$ as in \S~\ref{s},
	the rational number $\qq(0)=r/s$ in $C$ written in reduced form
	has the property that $b_1^{\ell}b_2^{\ell}$ divides $s$.
\end{lemma}

The proof is easy in certain cases, however for the general case
we need some preparation.
For convenient later referencing, we point out that if $p$ is a prime number and $e_1, e_2$ are rational numbers then
 \begin{equation} \label{eq:einfach}
     	v_p(e_1)\neq v_p(e_2) \qquad \Longrightarrow \qquad v_p(e_1+e_2)= \min \{ v_p(e_1), v_p(e_2) \}. 
 \end{equation}

The claim is well-known and we omit its short very elementary proof.

\begin{proposition}  \label{klaro}
	Let $f,g$ be strict contractions on $\R$,
	inducing different fixed points $f^{\infty}(0)\ne g^{\infty}(0)$. Then
	for any integer $N\ge 1$, we have $f^Ng^{\infty}(0)\ne g^{\infty}(0)$.
\end{proposition}

\begin{proof}
	It suffices to show the claim for all large enough $N$, as if
	there is equality $f^N g^{\infty}(0)= g^{\infty}(0)$ for some $N$, then
	the equality holds for all positive integer multiples of $N$ as well.
	It is easily checked that for any element $e\in \R$
	the sequence $f^N(e)$ tends to the fixed point $f^{\infty}(0)$ 
	of $f$ as $N\to\infty$. Hence 
	if the fixed points of $f, g$ are different and thus have positive
	distance, for large enough $N$ there is positive distance between
	$f^N(e)$ and the fixed point of $g$. It then suffices to take $e=g^{\infty}(0)$ the fixed point of $g$. 
	\end{proof}

The claim can be generalized to any metric space by the same argument.

\begin{proposition}  \label{Pro}
	Let $F(x)= x/b+r/s$ be any affine contraction with rational coefficients. Let $u/v\neq br/((b-1)s)$ be a rational number 
	not equal to the fixed point of $F$. Then for an arbitrarily large integer $t\ge 0$, for large enough $N\ge N_0(t)$ we have that $F^N(u/v)$ is a rational number and after reduction has
	denominator divisible by $b^{t}$.
\end{proposition}

\begin{proof}
	For a formal variable $x$ and an integer $N\ge 1$, it is easily checked that 
	\[
	F^N(x)= \frac{x}{b^N}+ \frac{r}{s}(1+b^{-1}+\cdots +b^{-N+1}).
	\]
	Inserting $x=u/v$ and simplifying we get
	\begin{equation}  \label{eq:denom}
	F^N(u/v)= \frac{1}{b^N}\cdot \frac{ us(b-1)+rbv(b^N-1)}{vs(b-1)}.
	\end{equation}
	Denote the numerator by
	\[
	H_N= us(b-1)+rbv(b^N-1).
	\]
	Take $N\ge t$ and $p$ any prime divisor of $b$.
	Reducing $H_N$
	modulo $p^{t}$ gives residue class $us(b-1)-rbv$, which is a constant independent of $N$ and $p$.
	Our assumption $u/v\neq br/((b-1)s)$ is equivalent to
	this constant being non-zero. Then 
	$z_p:= v_p(us(b-1)-rbv)<\infty$ is finite and 
	we may assume without loss
	of generality that $t>\max_{p|b} z_p$, as increasing $t$ sharpens
	the claim we aim to prove. 
	Then by \eqref{eq:einfach} we infer
	\[
	v_p(H_N) \leq \max_{p|b} z_p < t\le t v_p(b),\qquad p|b.
	\]
	Since the denominator in \eqref{eq:denom} in given form 
	contains every such prime factor dividing $b$ 
	at least $Nv_p(b)\ge N$ times, 
	after reduction of \eqref{eq:denom} there are still at least
	$(N-t)v_p(b)$ factors $p$ left in the denominator. Since this holds for 
	any $p|b$, we deduce that
	$b^{N-t}$ divides the reduced denominator. Thus 
	it suffices to take $N=2t$.
	\end{proof}

Given an IFS as in \eqref{eq:DORN}, we introduce classes of prime numbers by
\[
\mathscr{P}_1= \{ p\in \mathbb{P}: \; p|b_1 \}, \quad \mathscr{P}_2= \{ p\in \mathbb{P}: \; p|b_2 \}, \quad \mathscr{P}_3=\mathscr{P}_1\setminus \mathscr{P}_2= \{ p\in \mathbb{P}: \; p|b_1,\; p\nmid p_2 \}, 
\]
and finally let
\begin{equation} \label{eq:HHh}
\mathscr{P}_4= \mathscr{P}_1 \cup \mathscr{P}_2= \mathscr{P}_3 \cup \mathscr{P}_2
\end{equation}
be the prime divisors of $b_1b_2$. We can finally prove our lemma.
We will denote by $\rm{rad}(N)$ the radical of a positive integer $N$,
defined as the product of its prime factors, without multiplicities. 

\begin{proof}[Proof of Lemma~\ref{lemur}]
	We first apply Proposition~\ref{Pro} for
	\[
	t=2\ell, \qquad
	u/v= g^{\infty}(0), \qquad F=f, \qquad b=b_1,
	\]
	to get that  for some large enough integer $N_1$ the integer $b_1^{2\ell}$ divides the denominator of the rational number $f^{N_1}g^{\infty}(0)$
	after reduction. Write $A/B=f^{N_1}g^{\infty}(0)$ in reduced form. 
	We may increase $\ell$ and hence
	$N_1$ if necessary so that all primes in $\mathscr{P}_1$
	occur in $B$ in strictly higher multiplicity
	than in $s_2$. Thus we have 
	\begin{equation}  \label{eq:asua}
	v_p(B)\ge \max\{ 2\ell v_p(b_1) , v_p(s_2)+1\}, \qquad p\in \mathscr{P}_1. 
	\end{equation}
	Moreover $f^{N_1}g^{\infty}(0)\neq g^{\infty}(0)$ 
	by Proposition~\ref{klaro}.
	Then we again apply Proposition~\ref{Pro} with
	\[
	t=2\ell, \qquad u/v= f^{N_1}g^{\infty}(0), \qquad F=g, \qquad b=b_2,
	\]
	which gives
	that for large enough $N_2$ the denominator of $r/s=g^{N_2}f^{N_1}g^{\infty}(0)$ in reduced form is divisible by 
	$b_2^{2\ell}$. In particular, we have
	\begin{equation}  \label{eq:J1}
	v_p(s)\ge 2\ell v_p(b_2), \qquad p\in \mathscr{P}_2.
	\end{equation}
	Finally let $p\in \mathscr{P}_3$. Then by \eqref{eq:asua}, 
	$\mathscr{P}_3\subseteq \mathscr{P}_1$ and \eqref{eq:einfach} applied iteratively with
	\[
	e_1=\frac{r_2}{s_2}, \qquad e_2= \frac{g^{i-1}(A/B)}{b_2}
	\]
	for $1\le i\le N_2$,
	corresponding to $g^{N_2}(A/B)=r/s$, we easily see that the multiplicity of any $p\in \mathscr{P}_3$ in the reduced denominator does not decrease (in fact remains equal). So,
	in short, \eqref{eq:asua} implies
	\begin{equation}  \label{eq:J2}
	v_p(s)\ge v_p(B)\ge 2\ell v_p(b_1), \qquad p\in \mathscr{P}_3.
	\end{equation}
	By \eqref{eq:HHh}, combining \eqref{eq:J1}, \eqref{eq:J2} we see that
	\[
	v_p(s)\ge 2\ell \max\{  v_p(b_1), v_p(b_2)\}, \qquad p\in \mathscr{P}_4.
	\] 
	Hence $(\rm{rad}(b_1b_2))^{2\ell}$ divides $s$,
	and since $b_1b_2$ divides $(\rm{rad}(b_1b_2))^2$,
	the claim $(b_1^{\ell}b_2^{\ell})|s$ follows immediately.
	Finally, it is easy to see that
	we may increase $N_1$ or $N_2$ if necessary so that we can take 
	$N=\max \{N_1, N_2\}$.
	\end{proof}

%

\subsection{Rational approximations to $\xi_j$ and estimates }  \label{2.4}

In this paragraph we define sequences of rational approximations $p_{j,k}/q_{j,k}$
to $\xi_j$ from~\S~\ref{s}, and estimate their denominators (in reduced form) and how close they are to $\xi_j$. We will ultimately
obtain integers $q$ with small evaluations of $\Vert q\ux\Vert$ from these
sequences.
The next easy result will be involved in 
our crucial Corollary~\ref{ROC} below.

\begin{proposition}  \label{ppp}
	Let $u/v$ be any rational number. 
	Let $\bold{h}$ be any finite word on the alphabet $\{f,g\}$ with in total $h_1$ occurrences of digit $f$ and $h_2$ occurrences of 
	digit $g$,
	so that $|\bold{h}|=h_1+h_2$. Then
	the number $\bold{h}(u/v)$ is rational and after reduction
	has denominator dividing $b_1^{h_1}b_2^{h_2}\cdot s_1s_2v$.	
\end{proposition}

\begin{proof}
	For a formal variable $x$, formally applying the $|\bold{h}|$ contractions in a row,
	 each time either $h^{(i)}=f(x)=x/b+r_1/s_1$ or $h^{(i)}=g(x)=x/b+r_2/s_2$
	 and simplifying to standard fraction form,
	 an easy inductive argument yields that
	 \[
	 \bold{h}(x)= \frac{x+M}{b_1^{h_1}b_2^{h_2} s_1s_2},
	 \]
	 for some integer $M$, depending on the IFS and $\bold{h}$.
	When $x=u/v$ is rational then obviously the outcome $\bold{h}(x)$ can be written 
	as a rational number $(u+vM)/(b_1^{h_1}b_2^{h_2}s_1s_2v)$, so after reduction
	the denominator divides $b_1^{h_1}b_2^{h_2}s_1s_2 v$.
	\end{proof}

The next claim directly implies
that the rational numbers $\w_{j,k}(0)$ constructed
in \S~\ref{s} are almost reduced in the form as obtained
from formal symbolic computation. This is required for the proof of the lower bound for the Dirichlet constant in~\S~\ref{ASI}.

\begin{lemma}  \label{wicht}
	With the assumptions of Lemma~\ref{lemur},
	choose an integer $\ell$ satisfying
	\[
	\ell > \max_{i=1,2} \max_{p\in\mathbb{P},\; p|b_i} \frac{\max \{ v_p(s_1), v_p(s_2) \} }{v_p(b_i)}.
	\]
	Given such $\ell$, take $N$
	and let $\qq=g^Nf^N g^{\infty}$ as in Lemma~\ref{lemur}.
	Let $\bold{h}$ and $h_1, h_2$ be as in Proposition~\ref{ppp}.
	Then, the number 
	$\bold{h}\qq(0)\in C$ is rational and written in reduced form has denominator divisible by $b_1^{h_1}b_2^{h_2}$.
\end{lemma}

\begin{proof}
	We have to show that not much
	cancellation occurs when formally expanding $\bold{h}\qq(0)$
	as a rational number. 
	Start with
	\begin{equation}  \label{eq:rdurchs}
	\qq(0)=r/s 
	\end{equation}
	written in reduced form. By choice of $\ell$, the integer $s$ contains any
	prime divisor of $b_1b_2$ (i.e. in class $\mathscr{P}_4$ from~\S~\ref{2.1}) 
	more often than $s_1$ and $s_2$. Thus
	a very similar argument as in the proof
	of Lemma~\ref{lemur}, based on iterative application of \eqref{eq:einfach}, shows the following: 
	Let $\bold{h}=h^{(1)}h^{(2)}\cdots h^{(|\bold{h}|)}$.
	If we read the digit $h^{(i)}=f$ at some position $i$, we get another factor $b_1$ in the reduced denominator while the 
	multiplicity of remaining prime factors of $b_2$ (not dividing $b_1$) cannot decrease, in fact remain equal.
	Vice versa, reading $h^{(i)}=g$ instead gives a factor $b_2$ in the reduced denominator whereas the multiplicity of remaining primes 
	dividing $b_1$ but not $p_2$ (i.e. in class $\mathscr{P}_3$ from~\S~\ref{2.1}) cannot decrease. We omit the details. 
	Repeating this argument for $i=|\bold{h}|, |\bold{h}|-1, \ldots, 1$ 
	to transform $r/s$ into $\bold{h}(r/s)=\bold{h}\qq(0)$,
	and by definition of $h_1, h_2$,
	the claim follows directly.
	\end{proof}

\begin{remark}
	The proof in fact yields that $sb_1^{h_1}b_2^{h_2}$ divides the reduced denominator.
\end{remark}

In the sequel we always let $r, s$ be given as in \eqref{eq:rdurchs},
and coprime, and $N$ as in Lemma~\ref{wicht}. 
Derive the words $\pp= f^Ng^{\infty}$ and $\qq=g^N f^N g^{\infty}$
and $\ttt_{j,k}, \vv_{j,k}, \w_{j,k}, \w_j$, the 
integers $\textswab{f}_k, \textswab{g}_{j,k}$
and finally $\ux$ as 
in~\S~\ref{s}.

\begin{definition}
	For $1\leq j\leq m, k\ge 1$ and $i\ge 1$,
	let $\tau_{j,k,i}\in \{ b_1^{-1},b_2^{-1}\}$ be the contraction
	factor induced by the $i$-th digit $w_{j,k,i}\in\{f,g\}$ of $\w_{j,k}$. 
\end{definition}

It follows from the observation concluding~\S~\ref{s} that
\begin{equation}  \label{eq:dieid}
P_{j,k}:=\prod_{i=1}^{|\ttt_{j,k}|} \tau_{j,k,i}^{-1}=  b_1^{\textswab{f}_{k}}b_2^{\textswab{g}_{j,k} }\in \mathbb{Z},\qquad
1\leq j\le m,\; k\ge 1.
\end{equation}
Then $P_{j,k}$ corresponds to the prefix factor for the denominator 
of $\w_{j,k}(0)$, and essentially plays the role of the integer $a_{(k-1)m+j}$ with the notation in~\cite{js} for the special case of missing digit fractals.
For simplicity, let
\[
P_{k}:= P_{1,k}, \qquad P_k^{\ast}:= P_{k,m}.
\]
Then \eqref{eq:ple}, \eqref{eq:simply} are equivalent to
\begin{equation}  \label{eq:ufoA}
P_{j,k}\asymp P_{k}^{j}, \quad (1\leq j\leq m-1),\qquad \quad P_{k}^{\ast}=P_{m,k}\asymp c^m P_{k}.
\end{equation}

Recall $\w_{j,k}=\vv_{j,k}\pp$ are ultimately periodic
words, for $1\le j\le m, k\ge 1$.
Hence $\w_{j,k}(0)\in \mathbb{Q}$, so in the sequel we write
\[
\w_{j,k}(0)= \vv_{j,k}\bold{p}(0)= \ttt_{j,k}\qq(0)= p_{j,k}/q_{j,k}
\]
where we assume the right hand side fractions are reduced.
Combining Proposition~\ref{ppp} with Lemma~\ref{wicht} for $\bold{h}=\ttt_{j,k}$ so that $\bold{h}\bold{q}(0)=\bold{h}(r/s)=\w_{j,k}(0)$, and
$u/v=r/s$, we
immediately get the following claim.

\begin{corollary}  \label{ROC}
	For $1\le j\le m$ and $k\ge 1$, we have 
	\[
	q_{j,k}=S_{j,k}P_{j,k},
	\]
	for integers $S_{j,k}$
	dividing the constant 
	\[
	S:= ss_1s_2.
	\]
	 In particular 
	\[
	q_{j,k}\asymp P_{j,k}, \qquad 
	1\leq j\leq m,\; k\ge 1,
	\] 
	with implied constants depending on the IFS only but not on $k$.
\end{corollary}

The following is easy to see.

\begin{proposition} \label{ap}
	We have the chain of divisibility 
	\[
	P_{1,k}| P_{2,k}| \cdots | P_{m,k}| P_{1,k+1}, \qquad k\geq 1.
	\]
\end{proposition}

\begin{proof}
	The claim in an obvious consequence of
	\[
	\textswab{f}_{k}\le\textswab{f}_{k+1}, \qquad \textswab{g}_{1,k}\le  \textswab{g}_{2,k}\le \cdots \le \textswab{g}_{m,k}\le \textswab{g}_{1,k+1}.
	\]
	These inequalities in turn follow from \eqref{eq:ID}, \eqref{eq:folgt}, \eqref{eq:simpli}, and the fast increase of the $M_i$.
\end{proof}


The next lemma determines, up to a constant, the distance from the rational approximations
$\w_{j,k}(0)=p_{j,k}/q_{j,k}$ to their limit $\xi_j$.

\begin{lemma}  \label{jo}
	With the above notation, we have
	\[
	|\xi_j- \frac{p_{j,k}}{q_{j,k}}| \asymp \prod_{i=1}^{|\vv_{j,k+1}|} \tau_{j,k,i} \asymp
	P_{j,k+1}^{-1}, \qquad\qquad 1\leq j\leq m,\; k\ge 1.
	\] 
	The implied constants depend on the IFS only but not on $k$. 
\end{lemma}


\begin{proof}
	By construction, for any $1\le j\le m$ and $k\ge 1$
	the words $\w_{j,k}$ and $\w_{j,k+1}$ agree
	up to (including) position $|\vv_{j,k+1}|$.
 Denote the suffixes of $\w_{j,k}$ resp. $\w_{j,k+1}$ starting
 from position $|\vv_{j,k+1}|+1$ 
 by $\boldsymbol{\sigma}_{j,k}$ resp.
 $\boldsymbol{\nu}_{j,k}$. By \eqref{eq:HH0}, these are given by
 \[
 \boldsymbol{\sigma}_{j,k}= g^{\infty}, \qquad  \boldsymbol{\nu}_{j,k}=\bold{p}= f^Ng^{\infty}.
 \]
 Note that this is independent of $j, k$, hence the distance
 between the associated rationals $\boldsymbol{\sigma}_{j,k}(0)$
 and $\boldsymbol{\nu}_{j,k}(0)$
 is constant as well.
  By Proposition~\ref{klaro} it is non-zero, thus
	 \begin{equation}  \label{eq:nonz}
	  |\boldsymbol{\sigma}_{j,k}(0)-\boldsymbol{\nu}_{j,k}(0)|
   \asymp 1.
	 \end{equation}
	
	Now since the digits of $\w_{j,k}$ and $\w_{j,k+1}$ agree
	up to position $|\vv_{j,k+1}|$, starting from $\boldsymbol{\sigma}_{j,k}$ and $\boldsymbol{\nu}_{j,k}$
	and reading consecutively the digits at places $|\vv_{j,k+1}|, |\vv_{j,k+1}|-1, \ldots,1$ to derive at $p_{j,k}/q_{j,k}$ resp. $p_{j,k+1}/q_{j,k+1}$,
	 we read the same contractions for both numbers $p_{j,k}/q_{j,k}$ and $p_{j,k+1}/q_{j,k+1}$. Hence the distance decreases by the contraction factor $1/b_1$ resp. $1/b_2$ in each step, depending on if we read for both $f$ or for both $g$. 
	Now
	by \eqref{eq:AB1}, the word $\vv_{j,k+1}$ consists of $\textswab{f}_{k+1}$ occurrences
	of $f$ and $\textswab{g}_{j,k+1}+N$ occurrences of $g$. Hence by \eqref{eq:nonz} and as $N$ is fixed,
	the total distance between two consecutive rationals becomes
	\begin{align*}
	|\frac{p_{j,k}}{q_{j,k}}-\frac{p_{j,k+1}}{q_{j,k+1}}|&= |\w_{j,k}(0)-\w_{j,k+1}(0)|\\
	&=|\vv_{j,k+1}\boldsymbol{\sigma}_{j,k}(0)-\vv_{j,k+1}\boldsymbol{\nu}_{j,k}(0)|    \\
	&= \prod_{i=1}^{|\vv_{j,k+1}|} \tau_{j,k,i}\cdot |\boldsymbol{\sigma}_{j,k}(0)-\boldsymbol{\nu}_{j,k}(0)|  \\
	&= b_1^{- \textswab{f}_{k+1} } b_2^{-\textswab{g}_{j,k+1}-N}\cdot |\boldsymbol{\sigma}_{j,k}(0)-\boldsymbol{\nu}_{j,k}(0)|\\
	&\asymp b_1^{- \textswab{f}_{k+1} } b_2^{-\textswab{g}_{j,k+1}-N}\\
	&\asymp b_1^{- \textswab{f}_{k+1} } b_2^{-\textswab{g}_{j,k+1}}\\
	&=P_{j,k+1}^{-1}.  
	\end{align*}
	Finally as $\xi_j$ is the limit of $p_{j,k}/q_{j,k}$
	as $k\to\infty$, we conclude
	\[
	|\xi_j- \frac{p_{j,k}}{q_{j,k}}|= \lim_{u\to\infty}
	|\frac{p_{j,k}}{q_{j,k}}-\frac{p_{j,k+u}}{q_{j,k+u}}|=
	|\sum_{i=k}^{\infty} (\frac{p_{j,i}}{q_{j,i}}-\frac{p_{j,i+1}}{q_{j,i+1}})|
\asymp
	|\frac{p_{j,k}}{q_{j,k}}-\frac{p_{j,k+1}}{q_{j,k+1}}| \asymp P_{j,k+1}^{-1}
	\]
	since we may assume the lengths $|\vv_{j,k}|$ grow rapidly with $k$
	by choosing $(M_i)_{i\ge 1}$ fast increasing,
	so the first term in the sum is dominating.
	\end{proof}

We can now finally present the core of the proof, that is to show $\Theta(\ux)\asymp c$ for $\ux$ constructed above.
All implied constants will be understood to depend on the IFS only.
We will assume
\[
m\ge 3
\] 
in the proof below. Otherwise if $m=2$
we have to slightly modify the construction from~\S~\ref{s} according to~\cite[\S~7]{js},
concretely we may alter $\textswab{g}_{m,k}=\textswab{g}_{2,k}$ from \eqref{eq:simply}, \eqref{eq:simpli} to
\[
(b_1^{ \textswab{f}_{k}-\textswab{f}_{k-1} } b_2^{ \textswab{g}_{2,k}-\textswab{g}_{1,k} })^{2} \asymp c^{-2}b_1^{  \textswab{f}_{k} } b_2^{ \textswab{g}_{2,k} }, \qquad
\textswab{g}_{2,k}= 
\left\lfloor \frac{ \log(c^{-2}b_1^{ (k-2)N }b_2^{2\textswab{g}_{1,k} })  }{ \log b_2 }   \right\rfloor,
\]
where we used \eqref{eq:ID} in the right hand side.
While the preparatory results
from~\S~\ref{2.1},~\ref{2.4} above are unaffected, 
a few modifications in~\S~\ref{SIA},~\ref{ASI} below occur
but can be handled very similar to~\cite[\S~7]{js}, we omit the details.
We finally assume 
\begin{equation}  \label{eq:Tt}
\frac{\log M_{i+1}}{\log M_i} \to\infty,
\end{equation}
for the sequence $(M_i)_{i\ge 1}$ of~\S~\ref{s}.
Recall $S$ defined in Corollary~\ref{ROC} for the proof below.

\subsection{Proof $\Theta(\ux)\ll c$ and Liouville property }   \label{SIA}

Combining \eqref{eq:Tt} with Lemma~\ref{jo}, 
we get that $P_{k+1}$ is much larger than $P_k$, more precisely
\begin{equation}  \label{eq:pkk}
P_{k}=P_{k+1}^{o(1)}, \qquad k\to\infty.
\end{equation}
Let $Q>1$ be large.
Let $P_0:=1$ and $k\ge 0$ be the unique integer such that
\[
S \cdot P_{k}\leq Q < S \cdot P_{k+1}.
\]

Case 1: $Q< S \cdot P_k^{\ast}$. Then $Q\ll P_k^{\ast} \asymp c^m P_k^{m}$ by \eqref{eq:ufoA}. Let
\[
q=S \cdot P_k\asymp P_k.
\]
Then $q\cdot (p_{1,k}/q_{1,k})=S P_k\cdot (p_{1,k}/q_{1,k})$ as well as
$q\cdot (p_{j,k-1}/q_{j,k-1})=S P_k \cdot (p_{j,k-1}/q_{j,k-1})$ 
for $2\leq j\leq m$ are integers by Corollary~\ref{ROC} (Proposition~\ref{ppp} suffices) and Proposition~\ref{ap}.
For $j=1$, by Lemma~\ref{jo} and \eqref{eq:pkk}, for arbitrarily large 
$t>0$ and $k\geq k_0(t)$ we get
\[
\Vert q\xi_1\Vert\le \Vert q\xi_1- \frac{ qp_{1,k}}{q_{1,k}}\Vert
\le q|\xi_1-\frac{p_{1,k}}{q_{1,k}}| \ll P_kP_{k+1}^{-1} \ll Q^{-t}.
\]
For $j>1$, we first combine Lemma~\ref{jo} and \eqref{eq:ufoA} to get
\[
|\xi_j-\frac{p_{j,k-1}}{q_{j,k-1}}| \asymp P_{j,k}^{-1}
\asymp P_k^{-j}, \qquad 2\leq j\leq m-1,\; k\ge 1
\]
and for $j=m$ similarly
\[
|\xi_m-\frac{p_{m,k-1}}{q_{m,k-1}}| \asymp P_{m,k}^{-1}
\asymp c^{-m} P_k^{-m}, \qquad k\ge 1.
\]
By combining this
with the bound for $Q$, for $2\le j\le m-1$ we get
\[
\Vert q\xi_j\Vert\le \Vert q\xi_j- \frac{ qp_{j,k-1}}{q_{j,k-1}}\Vert\le q|\xi_j-\frac{p_{j,k-1}}{q_{j,k-1}}|  \ll P_k
P_k^{-j}\ll P_k^{-1}\ll cQ^{-1/m},
\]
and similarly for $j=m$ by 
our assumption $m\ge 3$ we get
\[
\Vert q\xi_m\Vert\le \Vert q\xi_m- \frac{ qp_{m,k-1}}{q_{m,k-1}}\Vert\le q|\xi_m-\frac{p_{m,k-1}}{q_{m,k-1}}|  \ll P_k
c^{-m}P_k^{-m}\ll P_k^{-1}\ll cQ^{-1/m}.
\]
Thus indeed
\[
\Vert q\ux\Vert= \max_{1\le j\le m} \Vert q\xi_j\Vert \ll  cQ^{-1/m },
\]
as desired.


Case 2: $Q\ge  S \cdot P_k^{\ast}$. Here a crude estimation suffices. Let
\[
q= S\cdot P_k^{\ast} \asymp P_k^{\ast}.
\]
Then again by Corollary~\ref{ROC} and Proposition~\ref{ap}, all $q\cdot (p_{j,k}/q_{j,k})$ for $1\leq j\leq m$ are integers.
From \eqref{eq:pkk} and
from Lemma~\ref{jo}
for $\varepsilon>0$ and large indices $k\ge k_0(\varepsilon)$ 
that
\[
 \max_{1\leq j\leq m} \vert \xi_j- \frac{ p_{j,k}}{q_{j,k}}\vert \leq P_{k+1}^{-(1-\varepsilon/2)}
\]
and further by \eqref{eq:ufoA}, \eqref{eq:pkk} and since $m>1$ we conclude for arbitrary $\varepsilon>0$ and large $Q$ that
\begin{align*}
\Vert q\ux\Vert&\le \max_{1\leq j\leq m} \Vert q\xi_j- \frac{ qp_{j,k}}{q_{j,k}}\Vert\le q \max_{1\leq j\leq m} |\xi_j-\frac{p_{j,k}}{q_{j,k}}| \\ &\ll P_{k}^{\ast}\cdot P_{k+1}^{-1+\varepsilon/2}\ll c^m P_k^{m}P_{k+1}^{-1+\varepsilon/2} \\ &\ll P_{k+1}^{-1+\varepsilon}\ll Q^{-1+\varepsilon}\ll cQ^{-1/m}.
\end{align*}
Thus indeed $\Theta(\ux)\ll c$.
The Liouville property follows easily from the estimates
of case 2 and \eqref{eq:pkk}.

\subsection{Proof $\Theta(\ux)\gg c$}  \label{ASI}

Fix a large integer $k\ge 1$ and define
\[
Q=P_k^{\ast}-1.
\]
We will show that for every integer $1\le q\le Q$, we have 
\begin{equation}  \label{eq:juchuu}
\Vert q\ux\Vert \gg P_k^{-1} \gg cQ^{-1/m}, 
\end{equation}
where the right estimate comes from \eqref{eq:ufoA}. 
Once this is shown, the claim is obvious.

By Corollary~\ref{ROC}, we may write
$q_{j,k}=P_{j,k} S_{j,k}$ for $1\leq j\le m$, for $S_{j,k}$
divisors of $S$, hence again $|S_{j,k}|\ll 1$, and the fraction $p_{j,k}/q_{j,k}$ is reduced. Let $P_{0,k}:=1$.
Now for a given positive integer $q$, let $h=h(q)$ be the largest integer so that $P_{h,k}|q$,
which is well-defined and by our choice of $Q$ and Proposition~\ref{ap} 
satisfies $0\le h<m$.
We may then write $q= tP_{h,k}$ for $t$ an integer 
so that $t\nmid (P_{h+1,k}/P_{h,k})$, where the latter ratio 
is an integer
by Proposition~\ref{ap}.
Assume first $h<m-1$. Then
\[
q\cdot \frac{p_{h+1,k}}{q_{h+1,k}}= \frac{ tp_{h+1,k}\cdot P_{h,k}  }{ S_{h+1,k}P_{h+1,k} }
\]
is a rational number with denominator dividing
$P_{h+1,k}/P_{h,k}$ in lowest terms. 
Moreover it is not an integer
by $t\nmid (P_{h+1,k}/P_{h,k})$ and as $(p_{h+1,k},q_{h+1,k})=1$.
It follows from \eqref{eq:ufoA} that
\[
\Vert q\cdot \frac{p_{h+1,k}}{q_{h+1,k}}\Vert\ge S_{h+1,k}^{-1}\cdot \frac{ P_{h,k} }{ P_{h+1,k}}
 \gg \frac{ P_{h,k} }{ P_{h+1,k}}
\gg P_k^{-1}.
\]
Since the error term
\[
q|\xi_{h+1}-\frac{p_{h+1,k}}{q_{h+1,k}}|\ll qq_{j,k+1}^{-1}\ll Qq_{j,k+1}^{-1}\ll P_k^{\ast} P_{k+1}^{-1}
\] 
is of smaller order $o(P_k^{-1})$ by Lemma~\ref{jo} and \eqref{eq:pkk}, again by \eqref{eq:ufoA} we also have
\[
\Vert q\ux\Vert= \max_{1\leq j\leq k} \Vert q\xi_j\Vert
\geq \Vert q \xi_{h+1}\Vert\ge \Vert q\cdot \frac{p_{h+1,k}}{q_{h+1,k}}\Vert- q|\xi_{h+1}-\frac{p_{h+1,k}}{q_{h+1,k}}|
 \gg P_k^{-1}\gg cQ^{-1/m}.
\]
A similar argument applies for $h=m-1$. 
Here
again by \eqref{eq:ufoA} we even get the stronger
estimate
\[
\Vert q\cdot \frac{p_{m,k}}{q_{m,k}}\Vert
= \frac{ tP_{m-1,k}  }{ S_{m,k}P_k^{\ast }}\geq 
S_{m,k}^{-1 } \frac{P_{m-1,k}}{P_{m,k}}\gg \frac{P_{m-1,k}}{P_{m,k}}
\gg P_k^{-1},
\]
where we used that the expression 
is non-zero because $t< P_{m,k}/P_{m-1,k} = P_k^{\ast}/P_{m-1,k}$ and $(p_{m,k}, q_{m,k})=1$.
Hence \eqref{eq:juchuu} holds and consequently
$\Theta(\ux)\gg c$.

\section{Proof of Theorem~\ref{mull} }

First consider $\omega\in [\frac{1}{m}, \frac{1}{m-1}]$.
We first explain the deduction in the special case of Theorem~\ref{1}.
The main twist is to redefine the length of $\vv_{m,k}$ as
$|\vv_{m,k}|= \lfloor M_k/\omega\rfloor$, where we apply the 
small twist $|\vv_{m,k}|= \lfloor M_k/\omega\rfloor+k$
if $\omega=\frac{1}{m-1}$
to guarantee that $\ux$ is totally irrational in this case.
Note that \eqref{eq:agha} follows.
We need to replace the asymptotics
\eqref{eq:ufoA} for $j=m$ by 
\begin{equation}  \label{eq:FRITZ}
P_{m,k}\asymp P_k^{1/\omega},
\end{equation}
which we achieve by altering \eqref{eq:ple}, \eqref{eq:simply} accordingly. 
Doing so consistently, we get $\wo(\ux)\ge \omega$ by the method
of~\S~\ref{SIA} and $\wo(\ux)\le \omega$ by the method
of~\S~\ref{ASI}, with minor adjustments, especially using $P_k^{-1}\asymp P_k^{\ast -\om}$, which is equivalent to \eqref{eq:FRITZ}.
If the IFS is as in \eqref{eq:u1} in general form, the construction is precisely as above, but some twists occur in the proof. 
The preliminary results up to including Proposition~\ref{ap} 
are proved analogously, upon letting 
\[
P_{j,k}:= b_1^{\textswab{f}_{k}}b_2^{\textswab{g}_{j,k} }\in \mathbb{Z},\qquad
1\leq j\le m,\; k\ge 1.
\]
On the other hand,
we now have a mismatch between $P_{j,k}$ and 
$\prod \tau_{j,k,i}^{-1}$ since the contraction factors 
now satisfy $\tau_{j,k,i}\in \{ u_1b_1^{-1}, b_2^{-1} \}$.
Therefore in place of Lemma~\ref{jo} we now get
\begin{equation}  \label{eq:yi}
 |\xi_j- \frac{p_{j,k}}{q_{j,k}}| \asymp \prod_{i=1}^{|\vv_{j,k+1}|} \tau_{j,k,i}\asymp u_1^{\textswab{f}_{k+1}}P_{j,k+1}^{-1}, \qquad 1\leq j\leq m,\; k\ge 1.
\end{equation}
However, it is easy to see from \eqref{eq:ID}, \eqref{eq:AB1} and  \eqref{eq:Tt} 
that $\textswab{f}_k\asymp k$ and $\textswab{g}_{j,k}\gg M_k$,
so $\textswab{f}_k= o( \textswab{g}_{j,k})$ as $k\to\infty$ again
by \eqref{eq:Tt}, constituting the fast increase of the $M_k$.
Hence for large $k$
the factor $u_1^{\textswab{f}_{k+1}}$ is negligible 
compared to $P_{j,k+1}$, in other words \eqref{eq:yi} yields
\[
 |\xi_j- \frac{p_{j,k}}{q_{j,k}}| = P_{j,k+1}^{-1+o(1)}, \qquad k\to\infty.
\] 
This again suffices to obtain $\wo(\ux)\ge \omega$ by the method
of~\S~\ref{SIA}. The reverse
inequality is derived by the exact same line of arguments 
as in~\S~\ref{ASI} again,
replacing \eqref{eq:ufoA} by \eqref{eq:FRITZ}.

Now let $\om=1$.
Start with a fast increasing (lacunary) 
integer sequence $(a_i)_{i\ge 1}$ and slightly 
perturb it by adding $m$ absolutely bounded sequences $(v_i^{(j)})_{i\ge 1}$, $1\leq j\leq m$, $|v_i^{(j)}|\ll 1$, to form $m$ sequences $(a_i^{(j)})_{i\ge 1}$, $1\leq j\le m$. Then 
	take $\ux$ whose components $\xi_j$ have addresses  $g^{a_1^{(j)}}fg^{a_2^{(j)}}fg^{a_3^{(j)}}\cdots$. 
	The fast increase of the  $a_i^{(j)}$ yields $\wo(\ux)\ge 1$ by very
	similar arguments as in~\S~\ref{SIA} via looking basically at rational approximations of the form $g^{a_1^{(j)}}fg^{a_2^{(j)}}\cdots g^{a_n^{(j)}}fg^{\infty}(0)$, and the reverse estimate is trivial unless $\ux\in\mathbb{Q}^m$
 by \cite{khint}. By some inductive 
	variational argument, it can be 
	shown that certain choices of $v_i^{(j)}$ will guarantee
	that the arising real vector is totally irrational,
	similar to~\cite[Proposition~2.7]{ich2013}. See~\cite[\S~4]{ichmet}
	and corresponding proofs for more details in the special case of missing digit Cantor sets, which readily generalize to our setting.

\end{document}